\documentclass[11pt,reqno]{amsart}
\usepackage{calrsfs}
\usepackage[OT2,T1]{fontenc}
\usepackage[english]{babel}
\usepackage{times}

\usepackage{etoolbox}
\patchcmd{\section}{\scshape}{\bfseries}{}{}
\makeatletter
\renewcommand{\@secnumfont}{\bfseries}
\makeatother

\usepackage[margin=30mm,paperwidth=216mm, paperheight=303mm]{geometry}
\usepackage{fullpage}
\usepackage{parskip}
\usepackage{graphicx}
\usepackage{amssymb}
\usepackage{amsmath}
\usepackage{amsthm}
\usepackage{amsfonts}
\usepackage{mathtools}
\usepackage{mathrsfs}
\usepackage{stmaryrd}
\usepackage{thmtools}
\usepackage{hyperref}
\usepackage[shortlabels]{enumitem}
\usepackage[backend = bibtex, maxnames=5, citestyle=numeric, bibstyle=numeric]{biblatex}
\usepackage{tikz}
\usepackage{float}
\usepackage{diagbox}
\usepackage[labelformat=simple]{subfig}
\usepackage[export]{adjustbox}

\newcommand{\Z}{\mathbb{Z}}

\newcommand{\N}{\mathbb{N}}
\newcommand{\F}{\mathbb{F}}
\newcommand{\CC}{\mathbb{C}}
\newcommand{\R}{\mathbb{R}}

\newcommand{\Spec}{\textup{Spec}}

\newtheorem{tx}{Theorem}[section]
\newtheorem{txint}{Theorem}
\newtheorem*{txcob}{Cobham's Theorem}

\newtheorem{tc}{Corollary}[tx]
\newtheorem{tl}[tx]{Lemma}
\newtheorem{tp}[tx]{Proposition}

\theoremstyle{definition}
\newtheorem{td}[tx]{Definition}
\newtheorem*{te}{Example}
\newtheorem*{tes}{Examples}
\newtheorem*{tr}{Remark}
\newtheorem*{trs}{Remarks}

\bibliography{ref}
	
\date{\today}
\title{A graph-theoretic proof of Cobham's Dichotomy \\ for automatic sequences} 
\author[M. E.~Wessel]{Mieke Wessel} 
\address{\normalfont Mathematisch Instituut, Universiteit Utrecht, Budapestlaan 6, 3584 CD Utrecht, Nederland} 
\curraddr{\normalfont Mathematisches Institut, Georg-August Universit\"at, Bunsenstrasse 3-5, 37073 G\"ottingen, Deutschland}
\email{mieke.wessel@mathematik.uni-goettingen.de}

\begin{document}

	\begin{abstract} \noindent We give a new graph-theoretic proof of Cobham's Theorem which says that the support of an automatic sequence is either sparse or grows at least like $N^\alpha$ for some $\alpha > 0$. The proof uses the notions of tied vertices and cycle arboressences. With the ideas of the proof we can also give a new interpretation of the rank of a sparse sequence as the height of its cycle arboressence. In the non-sparse case we are able to determine the supremum of possible $\alpha$, which turns out to be the logarithm of an integer root of a Perron number. 
	\end{abstract}
\maketitle
\section{Introduction}
The objects of interest in this paper are automatic sequences. A sequence $(a_i)_{i\in \N}$ is called \textit{automatic} if it can be derived from an automaton. Given some positive integer $m$, an \textit{$m$-automaton} is a finite directed graph (or \textit{digraph}) where each vertex has exactly $m$ outgoing edges which are labeled by $0, 1, \ldots, m-1$. Furthermore, each vertex is labeled with an element from some set $A$ containing $0$ and one vertex is additionally labeled `Start'. To construct the value $a_i \in A$ out of this digraph, we take the base-$m$ expansion of $i$ and, beginning at `Start', follow the edges in the graph corresponding to its digits from back to front. The label of the vertex one ends up in defines the value of $a_i$. The precise definitions and some examples are given in Section \ref{Sdefaut}. 

Automatic sequences arise in many areas of mathematics, such as number theory, graph theory, formal languages and computer science. In this paper we will give a new graph-theoretic proof of the following theorem of Cobham, which was originally proved in the terminology of formal languages. 

\begin{txcob}[\label{Tcobhamdummy}\cite{cobham1972tagsequences}]
	Let $(a_i)_{i \in \N}$ be an automatic sequence. Then the support up to $N$, defined by $E((a_i)_{i \in \N})_N := \{ i \leq N \mid a_i \neq 0 \}$, meets one of the following two conditions:
	\begin{enumerate}[(i)]
		\item there exists some $\alpha \in (0, 1]$ such that $\# E((a_i)_{i \in \N}) \geq N^\alpha$ for sufficiently large $N$.
		\item there exists some $r \in \R_{\geq 0}$ such that $\# E((a_i)_{i \in \N})  = O(\log(N)^r)$. In this case we call $(a_i)_{i \in \N}$ sparse.
	\end{enumerate} 
\end{txcob}
(Caution: this is not the other `Cobham's Theorem', which says that a sequence is both $k$- and $m$-automatic for some multiplicative independent integers $k$ and $m$ if and only if it is also $1$-automatic or equivalently, ultimately periodic \cite{cobham1969basedependence}.)
%\begin{tr}
%	Often by Cobham's Theorem another theorem by Cobham is meant that says that a sequence is both $k$ and $l$-automatic for some multiplicative independent integers $k$ and $l$ if and only if it is also $1$-automatic \cite{cobham1969basedependence}. 
%\end{tr}

The advantages of the new proof are twofold. The first one comes from a renewed interest in the theorem from a number-theoretic perspective due to its relevance in study of the Nottingham groups \cite{byszewski2020automata}, also see the example in Section \ref{Sdefaut}. The second advantage is that the proof depends on recognizing which structures in the automaton lead to sparse or non-sparse behavior of the sequence. Using this analysis, we can also say which values of $r$ and $\alpha$ are in a sense optimal, see Theorem \ref{Tranksparsedummy} and \ref{Tnonsparsedummy}. The result for $r$ was proven earlier in \cite{szilard1992regularlanguages} using different methods and the bounds for both $r$ and $\alpha$ were proven by Schur in \cite{shur2008combcompreglang}. We use similar ideas as Schur, but in our proof, that we found independently, we streamline the argument using more graph theoretical tools. %Jason Bell pointed out that this can also be done in the way that I do it, should I mention this somewhere? See Remark 10.1.4 in main article

The proof we will give is based on an explicit characterization of automata that generate a sparse sequence, found separately by Ibarra and Ravikumar \cite{ibarra1986decisionproblems} and Szilard, Yu, Zhang and Shallit \cite{szilard1992regularlanguages}. The second characterization was first given in the formal language setting, but has since been translated to a graph-theoretic one by Byszewski and Konieczny \cite[Prop. 3.4]{byszewski2019autoseqgenpoly} and by Byszewksi, Cornelissen and Tijsma \cite[Prop. 12.1.2]{byszewski2020automata}. In this new formulation it states that a sequence is sparse if and only if its automaton does not contain any tied vertices. A vertex $v$ is called \textit{tied} if \emph{(i)} there exists a walk from $v$ to a non-zero labeled vertex and \emph{(ii)} there exist at least two different walks of the same length from $v$ to itself. 
%We will show that this characterization can also be proved by only looking at the graph. 

The starting point is to look at the correspondence between $\N$ and the walks in the automaton that end in a non-zero labeled edge, given by our algorithm to determine $a_i$. Evaluating $\# E_N$ then boils down to counting the number of walks that are associated with some $i \leq N$ and end in a non-zero labeled vertex. Hence, it will be relevant in how many ways we can reach some non-zero vertex within a bounded number of steps. This highly depends on whether or not there exists a tied vertex between the `Start'-vertex and the non-zero vertex. From this point of view it is relatively easy to find the exponential lower bound in the case of a tied vertex. The polylogarithmic upper bound when there are no tied-vertices is a bit harder and requires a more careful look at the structure of the automaton. 

One of the key ideas is that in the case of no tied vertices we may assume that the automaton has the structure of a cycle arboressence. Roughly, a \textit{cycle arboressence} is a directed tree where all the vertices that are not leaves are replaced by cycles and all the edges by smaller connected digraphs that do not contain any cycles (that is, the edges become actual directed trees). The leaves will be required to have only outgoing edges that are self-loops and should have label $0$. From each vertex that is not a leaf it should be possible to reach a non-zero vertex. See Definition \ref{Dcycarbo} and Figure \ref{Fcyclearbo} for a more precise depiction. For every non-zero vertex in this cycle arboressence, there is a unique sequence of cycles through which one must (partially) pass to reach it. This can be used to count the number of walks and give the polylogarithmic upper bound $O(\log(N)^r)$. Furthermore, the structure gives us a nice way to determine the smallest possible value for $r$. This value depends on the \textit{height} of the arboressence, defined as the maximum number of cycles one can pass through consecutively before reaching a leaf.

\begin{txint}\label{Tranksparsedummy}
	Let $(a_i)_{i \in \N}$ be a sparse automatic sequence. Then, there exists an automaton generating this series that is also a cycle arboressence. Write $r$ for the height of this cycle arboressence. 
	Then $$ \inf\{r' \in \R_{\geq 0} \mid \# E((a_i)_{i \in \N})_N = O(\log(N)^{r'})\} = r-1.$$
	Furthermore, the infimum is attained.
\end{txint}

To also determine the biggest value for $\alpha$ in the case of a tied vertex, we use spectral graph theory, which is also used in \cite{krieger2007growthratepolytime} and \cite{shur2008combcompreglang}. We start by focusing on the connected sub-digraphs of the automaton that consist solely of tied vertices, since we already know by Cobham's Theorem that the other parts of the graph contribute less. For each tied vertex $v$ we call the largest of these containing $v$ the \textit{strongly connected sub-digraph of the automaton at $v$}. The adjacency matrix and its maximal absolute eigenvalue, also called its \textit{spectral radius}, tell us something about the number of walks there are of certain length inside the sub-digraph and in particular with which exponent the number of walks grows. This leads to the following theorem. 

\begin{txint}\label{Tnonsparsedummy}
	Let $(a_i)_{i \in \N}$ be a non-sparse automatic sequence, and let $V'$ be the subset of tied vertices that occur in some automaton producing $(a_i)_{i \in \N}$. Writing $\rho_v$ for the spectral radius of the strongly connected sub-digraph at $v \in V'$, and $\beta := \max_{v \in V'}(\rho_v)$, we have
	$$\sup\{\alpha \in (0, 1] \mid \# E((a_i)_{i \in \N}) \geq N^\alpha \textup{ for all sufficiently large $N$}\} = \log_m(\beta).$$ 
\end{txint}
\begin{tr}
	It is not possible to say anything general about whether or not the supremum is attained; there exist examples for both cases, see Figure \ref{Fsupatt}.
\end{tr}

From Theorem \ref{Tranksparsedummy} one can immediately conclude that $r$ has to be an integer. By a result of Lind \cite{lind1983entropies}, we can also say something about the kind of number that $\beta$ can be. 
%\begin{td}
%	Let $\rho > 1$ be an algebraic integer in $\R$ such that all other roots of its minimal polynomial have smaller absolute value than $|\rho|$. Then $\rho$ is called a \textit{Perron number}.
%\end{td}
\begin{txint}\label{Tperrondummy}
	Let $\beta$ be as defined in Theorem \ref{Tnonsparsedummy}. Then $\beta$ is equal to an integer root of a Perron number. That is, it is an integer root of a real algebraic integer $\rho > 1$, such that all the algebraic conjugates of $\rho$ have absolute value smaller than $|\rho|$. 
\end{txint}

In Section 2 we will give all the basic notions and definitions. In Section 3 we will give the new proof of Cobham's Theorem and Theorem \ref{Tranksparsedummy}. In Section 4 we give a proof of Theorem \ref{Tnonsparsedummy} and Theorem \ref{Tperrondummy}.
\subsection*{Notation}~\\
By $\N$ we denote all positive integers and thus it does not contain $0$. 

By $[N]$ we denote all the natural number up to and including $N$, so $[N]:= \{1, 2, \ldots, N\}$. 

We recall that a walk $W$ in a graph is any sequence of edges $(e_1, \ldots, e_k)$ such that the end vertex of $e_i$ equals the start vertex of $e_{i+1}$ for all $1 \leq i \leq k-1$. In particular a walk can contain the same edge multiple times. If the start vertex of $e_1$ equals the end vertex of $e_k$ we call $W$ a closed walk. Note that a closed walk is not necessarily a cycle, which cannot contain the same vertex multiple times. 

If $W = (e_1, \ldots, e_k)$ is a walk in some graph, its length $l(W)$ is defined as the number $k$, that is the number of edges $W$ contains including possible repeated edges. 
\section{Prerequisites}\label{CHdef}
\subsection{Automatic Sequences} \label{Sdefaut} As mentioned before, to obtain automatic sequences we first need finite automata. 
\begin{td}
	Let $m$ be a natural number and $A$ a finite set containing $0$. An \textit{$m$-automaton with output in $A$} is a finite directed graph for which the following hold:
	\begin{itemize}
		\item each vertex is labeled by an element of $A$. 
		\item one vertex is additionally labeled with `Start'. 
		\item each vertex has $m$ outgoing edges, that are all labeled with a different element of $\{0, 1, \ldots, m-1\}$.
		\item an edge labeled with $0$ connects two vertices with the same label.  (This is called \textit{leading zero invariance}.)
		\item each vertex can be reached from the start vertex. (This is called \textit{accessibility}.)
	\end{itemize}
\end{td}

\begin{td}
	Let $k \in \Z$, $m$ a natural number and $(b_i)_{i\geq 0}$ a (finite) base-$m$ expansion of $k$, in symbols $k = \sum_{i=0}^nb_im^i$. Suppose we are given an $m$-automaton with states in some set $A$. From the `Start' vertex consecutively follow the edges labeled by $b_0, b_1, \ldots, b_n$ and suppose we end up at vertex $v$. Define $a_k$ to equal the label of $v$ in $A$. Then $(a_i)_{i\geq 0}$ is called an \textit{$m$-automatic sequence}. 
\end{td}

\begin{tr}
	The fourth property of an $m$-automaton implies that an automatic sequence does not depend on which base-$m$ expansion one takes for an integer. For instance writing $5 = 1\cdot1 + 0\cdot2 + 1\cdot4$ will give the same value for $a_5$ as writing $5 = 1\cdot1 + 0\cdot2 + 1\cdot4+0\cdot 8$. This is where the name leading zero invariance comes from. 
\end{tr}
\begin{tes}
	An important class of examples are those for which $m = p$ and $A = \F_p$ for some prime $p$. We consider the two $2$-automata with states in $\F_2$ in Figure \ref{FautoF2} and write $(a_i)_{i \geq 0}$ and $(b_i)_{i\geq 0}$ respectively for the automatic sequences generated by the automata in \ref{FautoF2.a} and \ref{FautoF2.b}.
	\begin{itemize}
		\item Suppose we are interested in calculating $a_{13}$ and $b_{13}$. In binary $13 = 1101$, so we must follow the arrows labeled by $1, 0, 1$ and $1$ consecutively. In (a) we end up in the vertex at the bottom labeled with $0$ and in (b) in the top right vertex labeled with $1$. Hence, $a_{13} = 0 $ and $b_{13} = 1$. 
		\item Using this method for all $i \in \Z_{\geq 0}$ we get $(a_i) = (0, 1, 1, 0, 1, 0, 0, 0, 1, \ldots)$ and \\$(b_i) = (0, 1, 0, 0, 1, 0, 0, 0, 0, \ldots)$. 
	\end{itemize}
\end{tes}
\begin{tr}
	The case where $m=p$ and $A = \F_p$ is also the relevant one for the Nottingham groups. The Nottingham group $\mathcal{N}_p$ consists of formal power series of the form $t + O(t^2)$ with coefficients in $\F_p$ and is a group under composition. Christol showed that an element is algebraic over $\F_p(t)$ if and only if its sequence of coefficients is automatic \cite{christol1979autoalg}. It is also known that every element of finite order is conjugate to an algebraic, and thus automatic, series \cite[Remark 1.6]{bleher2016automcurves}. An open question is if every element of finite order in $\mathcal N_p$ is conjugate to a series such that the sequence of coefficients is sparse \cite{byszewski2020automata}. Here we can also see that Cobham's Theorem becomes relevant, since it suffices to show every such element is conjugate to an automatic series of which the support grows less fast than fractionally in $N$. 
\end{tr}

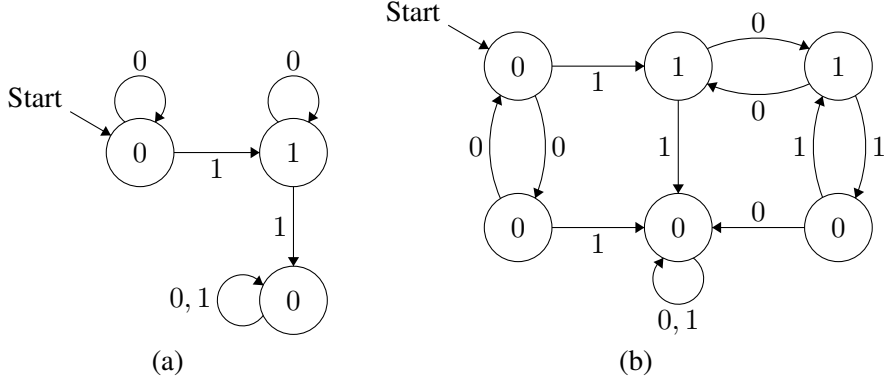
\begin{figure}[h] 
	\centering
	\subfloat[\label{FautoF2.a}]{\centering
		\begin{tikzpicture}[scale=0.15]
			\tikzstyle{every node}+=[inner sep=0pt]
			\draw [black] (15.2,-10.9) circle (3);
			\draw (15.2,-10.9) node {$0$};
			\draw [black] (28.8,-10.9) circle (3);
			\draw (28.8,-10.9) node {$1$};
			\draw [black] (28.8,-23.9) circle (3);
			\draw (28.8,-23.9) node {$0$};
			\draw [black] (13.877,-8.22) arc (234:-54:2.25);
			\draw (15.2,-3.65) node [above] {$0$};
			\fill [black] (16.52,-8.22) -- (17.4,-7.87) -- (16.59,-7.28);
			\draw [black] (18.2,-10.9) -- (25.8,-10.9);
			\fill [black] (25.8,-10.9) -- (25,-10.4) -- (25,-11.4);
			\draw (22,-11.4) node [below] {$1$};
			\draw [black] (28.8,-13.9) -- (28.8,-20.9);
			\fill [black] (28.8,-20.9) -- (29.3,-20.1) -- (28.3,-20.1);
			\draw (28.3,-17.4) node [left] {$1$};
			\draw [black] (27.477,-8.22) arc (234:-54:2.25);
			\draw (28.8,-3.65) node [above] {$0$};
			\fill [black] (30.12,-8.22) -- (31,-7.87) -- (30.19,-7.28);
			\draw [black] (26.12,-25.223) arc (324:36:2.25);
			\draw (21.55,-23.9) node [left] {$0,1$};
			\fill [black] (26.12,-22.58) -- (25.77,-21.7) -- (25.18,-22.51);
			\draw [black] (9,-7.3) -- (12.61,-9.39);
			\draw (8.35,-6.06) node [left] {Start};
			\fill [black] (12.61,-9.39) -- (12.16,-8.56) -- (11.66,-9.42);
	\end{tikzpicture}}
	\qquad
	\subfloat[\label{FautoF2.b}]{\centering
		\begin{tikzpicture}[scale=0.15]
			\tikzstyle{every node}+=[inner sep=0pt]
			\draw [black] (15.2,-10.7) circle (3);
			\draw (15.2,-10.7) node {$0$};
			\draw [black] (29.4,-10.7) circle (3);
			\draw (29.4,-10.7) node {$1$};
			\draw [black] (43.6,-10.7) circle (3);
			\draw (43.6,-10.7) node {$1$};
			\draw [black] (15.2,-24.9) circle (3);
			\draw (15.2,-24.9) node {$0$};
			\draw [black] (29.4,-24.9) circle (3);
			\draw (29.4,-24.9) node {$0$};
			\draw [black] (43.6,-24.9) circle (3);
			\draw (43.6,-24.9) node {$0$};
			\draw [black] (9,-7.1) -- (12.61,-9.19);
			\draw (8.35,-5.86) node [left] {Start};
			\fill [black] (12.61,-9.19) -- (12.16,-8.36) -- (11.66,-9.22);
			\draw [black] (18.2,-10.7) -- (26.4,-10.7);
			\fill [black] (26.4,-10.7) -- (25.6,-10.2) -- (25.6,-11.2);
			\draw (22.3,-11.2) node [below] {$1$};
			\draw [black] (31.967,-9.164) arc (113.37894:66.62106:11.424);
			\fill [black] (41.03,-9.16) -- (40.5,-8.39) -- (40.1,-9.31);
			\draw (36.5,-7.73) node [above] {$0$};
			\draw [black] (41.042,-12.251) arc (-66.35568:-113.64432:11.326);
			\fill [black] (31.96,-12.25) -- (32.49,-13.03) -- (32.89,-12.11);
			\draw (36.5,-13.7) node [below] {$0$};
			\draw [black] (29.4,-13.7) -- (29.4,-21.9);
			\fill [black] (29.4,-21.9) -- (29.9,-21.1) -- (28.9,-21.1);
			\draw (28.9,-17.8) node [left] {$1$};
			\draw [black] (40.6,-24.9) -- (32.4,-24.9);
			\fill [black] (32.4,-24.9) -- (33.2,-25.4) -- (33.2,-24.4);
			\draw (36.5,-24.4) node [above] {$0$};
			\draw [black] (18.2,-24.9) -- (26.4,-24.9);
			\fill [black] (26.4,-24.9) -- (25.6,-24.4) -- (25.6,-25.4);
			\draw (22.3,-25.4) node [below] {$1$};
			\draw [black] (16.708,-13.284) arc (22.86781:-22.86781:11.621);
			\fill [black] (16.71,-22.32) -- (17.48,-21.77) -- (16.56,-21.38);
			\draw (18.12,-17.8) node [right] {$0$};
			\draw [black] (13.635,-22.351) arc (-156.09794:-203.90206:11.233);
			\fill [black] (13.63,-13.25) -- (12.85,-13.78) -- (13.77,-14.18);
			\draw (12.17,-17.8) node [left] {$0$};
			\draw [black] (30.723,-27.58) arc (54:-234:2.25);
			\draw (29.4,-32.15) node [below] {$0,1$};
			\fill [black] (28.08,-27.58) -- (27.2,-27.93) -- (28.01,-28.52);
			\draw [black] (45.108,-13.284) arc (22.86781:-22.86781:11.621);
			\fill [black] (45.11,-22.32) -- (45.88,-21.77) -- (44.96,-21.38);
			\draw (46.52,-17.8) node [right] {$1$};
			\draw [black] (42.2,-22.255) arc (-159.00947:-200.99053:12.437);
			\fill [black] (42.2,-13.34) -- (41.45,-13.91) -- (42.38,-14.27);
			\draw (40.87,-17.8) node [left] {$1$};
	\end{tikzpicture}}
	\caption{Automata for the sequences (a) $(a_i) = (0, 1, 1, \ldots)$ and (b) $(b_i) = (0, 1, 0, \ldots)$. Figure (b) is taken from \cite[~Figure 1]{byszewski2019autoseqgenpoly}.} \label{FautoF2}
\end{figure}

\begin{td}\label{DEFsize}
	Let $W$ be a walk in some $m$-automaton, such that the edges of $W$ are consecutively labeled by $b_1, b_2, \ldots, b_n$. Then we define the \textit{size $|W|$ of $W$} by $$b_1 + b_2m + b_3m^2 + \ldots + b_{n}m^{n-1}.$$
\end{td}

\begin{trs}~\begin{itemize}
		\item Note the distinction between the \emph{length} $l(W)$ and the \emph{size} $|W|$ of a walk $W$. For example, the walk $W$ with $v_0$ equal to the `Start'-vertex that we use to determine $a_k$ has size $|W| = k$ and length $l(W) = \lfloor \log_m(k) + 1 \rfloor$, assuming no leading zeroes.
		\item	Given a vertex $v$ in an $m$-automaton and an integer $k$ there is a unique walk of size $k$ that starts in $v$ and ends in a non-zero labeled edge. This follows from the fact that every integer $k$ has a unique base-$m$ expansion with a non-zero leading coefficient. We will use this in various parts of the proof. 
	\end{itemize}
\end{trs}

\begin{td}
	Let $(a_i)_{i\in \N}$ be a sequence in $A$, where $A$ is a set containing $0$, then its \textit{support} is defined to be the set 
	$$E((a_i)_{i\in \N}) = \{k \in \N \mid a_k \neq 0 \}.$$
	Furthermore, we write $E((a_i)_{i\in \N})_N$ for the set $E((a_i)_{i\in \N})\cap[N].$ 
\end{td}

\begin{td}
	A sequence $(a_i)_{i\in \N}$ is ($r$-)\textit{sparse} if 
	$$\# E((a_i)_{i\in \N})_N = O(\log(N)^r)$$
	for some $r \geq 0$. The infimum of such $r$ is called the \textit{rank of sparseness of $(a_i)_{i\in \N}$.} In the case that $(a_i)_{i\in \N}$ is automatic, the corresponding $m$-automaton is also called ($r$)-sparse. 
\end{td}

\subsection{Spectral graph theory}
The idea of spectral graph theory is to identify properties of the graph by examining its adjacency matrix. Here we recall some basic definitions and theory that we will need for Section \ref{SBoundsnumberwalks}. 
\begin{td}
	Let $D $ be a finite digraph (i.e. directed graph) with set of vertices $ \{v_1, \ldots, v_n\}$. Define $a_{i, j}$ as the number of edges going from $v_i$ to $v_j$. The $n \times n$ matrix $A_D := (a_{i, j})_{1 \leq i , j \leq n}$ is called the \textit{adjacency matrix of $D$}.
\end{td}
\begin{tl} \label{Ladjmat}
	Let $A_D$ be the adjacency matrix of some (di)graph $D = (V, E)$ and consider $B := A^k$ for some integer $k$. Then $b_{i, j}$ is equal to the number of walks from $v_i$ to $v_j$ of length $k$. 
\end{tl} 

In the following we will assume all matrices to be taken over $\R$. 

\begin{td}
	Let $A$ be a matrix. We say $A$ is \textit{positive} if all entries are positive and that $A$ is \textit{non-negative} if all entries are non-negative. The definition for vectors is analogous. 
\end{td}

\begin{td}
	Let $A$ be a square matrix. The set of distinct eigenvalues of $A$ is called the \textit{spectrum of} $A$ and denoted by $\Spec(A)$. The number
	$$\rho(A) := \max_{\lambda \in \Spec(A)}|\lambda|$$ 
	is the \textit{spectral radius of} $A$. Here $|.|$ is the usual absolute value on $\CC$. 
\end{td}

\begin{td}\label{Dreducible}
	Let $A$ be a square matrix. We say $A$ is a \textit{reducible matrix} if there exists a permutation matrix $P$ so that
	$$PAP^{-1} = \begin{pmatrix}
		X & Y \\
		0 & Z
	\end{pmatrix},$$
	where $X, Y, Z$ and $0$ are non-trivial matrices and $X$ and $Z$ are square. \\ 
	If $A$ does not meet this condition, it is called an \textit{irreducible matrix}.
\end{td}

\begin{td}
	A digraph $D = (V, E)$ is called \textit{strongly connected} if for each pair of vertices $u, v$ there exits a walk from $u$ to $v$. It is called \textit{weakly connected} if there exists a walk from $u$ to $v$ or from $v$ to $u$.
\end{td}

\begin{tl} \label{Lstrconnirr}
	A digraph $D$ is strongly connected if and only of its adjacency matrix $A_D$ is irreducible.
\end{tl}
%\begin{proof}
%	Assume that $A_D$ is reducible and that the vertices have been permuted such that $A_D = \begin{pmatrix}
%		X & Y \\ 0 & Z
%	\end{pmatrix}.$ Suppose $X$ and $Z$ have dimensions $m_x \times m_x$ and $m_z \times m_z$, receptively, and note that $m_x + m_z = n$. The zero-block in $A_D$ tells us that there does not exist an edge from $v_i$ to $v_j$ for any $m_x + 1 \leq i \leq n$ and $1 \leq j \leq m_x$. Therefore, there cannot exist a path from $v_n$ to $v_1$ and $D$ is not strongly connected. \\
%	Now suppose $D$ is not strongly connected. There must be vertices $v$ and $u$ so that there is no path from $v$ to $u$. Define $V_z$ to be the set of vertices $w$ for which there exists a path starting at $v$ and ending at $w$ and define $V_x$ as the other vertices. Neither will be empty, since $v \in V_z$ and $u \in V_x$. Now permute the vertices in such a way that $v_i \in V_x$ and $v_j \in V_z$ implies $i < j$. One checks that $A_D$ can then be written as in Definition \ref{Dreducible} and is thus reducible.  
%\end{proof}

We will use Perron-Frobenius theory, that deals with the spectra of matrices that are either positive or non-negative and irreducible. The results are quite numerous and can be found for instance in \cite[Chapter~8]{meyer2000matrixanalysislinalg}. The following proposition assembles the parts needed in this paper. 

\begin{tp}\label{PPerronFrobenius}
	Let $A$ be any non-negative irreducible matrix. Then
	\begin{enumerate}[(i)]
		\item the spectral radius $\rho(A)$ of $A$ is positive and contained in $\Spec(A)$. Also, it is a simple eigenvalue, that is it has multiplicity one.
		\item there exists a positive right-eigenvector $\vec{x}$ corresponding to the eigenvalue $\rho(A)$. 
	\end{enumerate}
\end{tp}
\begin{proof}
	See \cite[Perron--Frobenius Theorem p.~673]{meyer2000matrixanalysislinalg}.
\end{proof}

\section{Bounds for the growth of automatic sequences} \label{CHautosparse}
\subsection{Identifying sparseness by the automaton}
We will prove that we can see whether an automatic sequence is sparse by looking solely at its automaton. Namely, the series will be sparse if and only if the automaton contains no tied vertices. 
%This is a criterion that was first proved by Szilard, Yu, Zhang and Shallit \cite{szilard1992regularlanguages} in the regular language setting and then reformulated to our setting by ... and .... 
\begin{td}
	Let $v$ be a vertex of some $m$-automaton. Then $v$ is called \textit{tied} if the following two properties hold:
	\begin{itemize}
			\item[(i)] there exists a (possibly empty) walk from $v$ to a vertex with a label other than $0$. 
			\item[(ii)] there exist two different walks of the same length from $v$ to itself. 
		\end{itemize}
	See Figure \ref{Ftiedvertex} for a pictorial representation. 
\end{td}
\begin{tr}
As we will see in the next two lemmas, the second condition can also be replaced by the following equivalent statement:
\begin{itemize}[(ii)']
	\item there exist two different walks from $v$ to $v$ that do not contain $v$ a third time.
\end{itemize}
\end{tr}
\begin{tl} \label{Lwalksize}
	Let $W_1$ and $W_2$ be two walks in an $m$-automaton of length $n$ and $k$ respectively. Suppose that the last vertex of $W_1$ equals the first vertex of $W_2$. Then these walks can be composed to a walk with size $$|W_1 \circ W_2| = |W_1| + m^{n}|W_2|.$$
\end{tl}
\begin{proof}
	That the walks can be composed is clear. Now suppose $W_1$ has edges labeled by $b_1, \ldots, b_n$ and $W_2$ has edges labeled by $c_1, \ldots, c_k$. Then $W_1 \circ W_2$ has edges labeled by $b_1, \ldots, b_n, c_1, \ldots, c_k$ and thus its size becomes \[b_1 + b_2m + \ldots b_nm^{n-1} + c_1m^n + \ldots + c_{k}m^{n+k-1} = |W_1| + m^{n}|W_2|. \qedhere \]
\end{proof}
\begin{tl}\label{Lvcyc}
	Let $v$ be a vertex from which there exists a walk to a non-zero vertex. Then $v$ is not tied if and only if there is at most one walk from $v$ to $v$ that does not contain $v$ a third time.
\end{tl}
\begin{proof}
	We first prove the right to left implication. Write $W$ for the (possible) walk from $v$ to $v$ that does not contain another $v$. Then, any walk from $v$ to $v$ has to equal $nW$ (composing $W$ with itself $n$ times) and thus there cannot exist two distinct walks from $v$ to $v$ of the same length, implying $v$ is not tied. The left to right implication we will prove by contradiction. So, suppose there are two distinct walks $W_1$ and $W_2$ starting and ending in $v$ that do not contain $v$ a third time. Because $v$ is not tied, these two walks cannot have the same length. Define the length of $W_i$ to be $x_i$ and assume without loss of generality that $x_1 < x_2$. Since $W_1$ and $W_2$ start and end in $v$, they can be composed. Consider the walks $W_1\circ W_2$ and $W_2 \circ W_1$. The length of both these walks is $x_1 + x_2$ and they also still start and end in $v$. Furthermore, they are not the same walk, since the first walk will have a vertex $v$ at spot $x_1$ and the second walk will not, because $W_2$ would then contain a third $v$. We have constructed two different walks from $v$ to $v$ of the same length. This contradicts the fact that $v$ is not tied and we may conclude that no such $W_1$ and $W_2$ can exist, which proves the lemma. 
\end{proof}

\begin{tx}\label{Tautosparse}
	Let $(a_i)_{i\in \N}$ be an automatic series in $A$, then one of the following must hold:
	\begin{itemize}
		\item the $m$-automaton contains a tied vertex and there is an $\alpha \in (0, 1]$ such that $\#E((a_i)_{i\in \N})_N \geq N^\alpha$ for sufficiently large $N$.  
		\item the $m$-automaton of $(a_i)_{i\in \N}$ contains no tied vertices and  $(a_i)_{i\in \N}$ is sparse. 
	\end{itemize}
\end{tx}
\begin{proof}
	First we will prove that if the $m$-automaton contains a tied vertex then we can find $\alpha > 0$ such that $\#E((a_i)_{i\in \N})_N \geq N^\alpha$ for sufficiently large $N$. Secondly we will show that if the $m$-automaton does not contain any tied vertices then $\#E((a_i)_{i\in \N})_N = O(\log(N)^r)$ for some $r \geq 0$, so that $(a_i)_{i\in \N}$ will be sparse. 
	
	Let $v$ be a tied vertex in an $m$-automaton that results in the automatic sequence $(a_i)_{i \in \N}$. By connectivity we know that there must be some walk $W_s$ from `Start' to $v$. Because $v$ is tied we also have distinct walks $W_1$ and $W_2$ going from $v$ to $v$ with equal length. Finally, we know that there is a walk $W_e$ from $v$ to a non-zero labeled vertex. (See Figure \ref{Ftiedvertex} for a pictorial representation.) We claim that we can assume that $W_e$ contains at least one edge and that its last edge is labeled non-zero. When $v$ is labeled by a $0$ this follows directly from the leading zero invariance and else, one may take $W_e$ equal to either $W_1$ or $W_2$ with all zero edges at the end removed, where we use that at least one of them does not consist solely of zero-edges. We write $z, x$ and $y$ for the lengths of $W_s, W_1$ and $W_e$ respectively. 
	\begin{figure}[h]
		\centering
		\includegraphics[width=5cm]{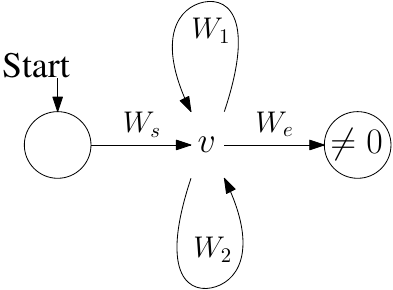}
		\caption{A tied vertex and its walks.} \label{Ftiedvertex}
	\end{figure}
	
	Now we can find a lower bound on the size of $E((a_i)_{i\in \N})_N$ using the tied vertex $v$. Namely, for each $k \in \Z_{\geq 0}$ the set \[\{|W_s \circ W_{i_1} \circ W_{i_2} \circ \ldots \circ W_{i_k} \circ W_e| : i_j \in \{1, 2\} \textup{ for } 1 \leq j \leq k \}\]
	is a subset of $E((a_i)_{i\in \N})$. Note here that we used that $W_s$ starts at `Start' and ends in $v$, that $W_1$ and $W_2$ both start at and end in $v$, and that $W_e$ starts at $v$ and ends in a non-zero vertex. Also notice that all choices of $k$ and $i_j$ give walks of different sizes, because the walks are different and do not end with a zero-edge. Hence, by Lemma \ref{Lwalksize}, for each $k$ we find in this way $2^k$ unique elements in $E((a_i)_{i\in \N})$ that all lie between $m^{z+y+(k-1)x}$ and $m^{z+y+kx}$. Even better, we find at least $2^k - 1$ elements in $E((a_i)_{i\in \N})_N$ for $N = m^{z+y+(k-1)x}$.
	
	Define $B := \frac{1}{x}\log_m(2)$, let $0<\alpha<B$ and $N$ a sufficiently large integer. Pick $k$ such that $N$ lies between $m^{z+y+(k-1)x}$ and $m^{z+y+kx}$. Using our previous observations we find that
	$$\#E((a_i)_{i\in \N})_N \geq 2^k - 1 = m^{xkB} -1 \geq (m^{z+y})^\alpha m^{xk\alpha} \geq (m^{z+y+xk})^\alpha \geq N^\alpha, $$
	where for the second inequality we use that $N$ (and therefore $k$) is sufficiently large so that $(B-\alpha)xk > (z+y)\alpha$. We have now shown that, if the $m$-automaton contains a tied vertex, then $\#E((a_i)_{i\in \N})_N \geq N^\alpha$ for some $\alpha$, proving the first part of the theorem. 
	
	To prove that $(a_i)_{i\in \N}$ being sparse implies the automaton has no tied vertices we will first show that we may assume that the $m$-automaton has a very specific structure, namely that of a cycle arboressence with as a root vertex the one labeled by `Start'. 

	\begin{td}\label{Dcycarbo}
		Let $D = (V, E)$ be a weakly connected finite digraph. We call $D$ a \textit{cycle arboressence} if there exists a $w \in V$ such that, if we add an extra edge $(\ast, w)$ to $E$ every vertex $v \in V$ is of one of the following types:
		\begin{enumerate}[(i)]
			\item all outgoing edges of $v$ are self-loops, and $v$ has exactly one other ingoing edge.
			\item any path starting at $v$ does not end at $v$, and $v$ has exactly one ingoing edge.
			\item there exists a unique path from $v$ to $v$, and on this path exactly one vertex has two ingoing edges, the others have one ingoing edge.
		\end{enumerate}
	 If a vertex falls in the category (.) it will be called a  \textit{type-(.) vertex} and additionally $w$ will be called a \textit{root vertex}. The \textit{height} of $D$ is defined as the maximum number of distinct cycles one can pass through consecutively.
	\end{td}
	\begin{trs}~
		\begin{itemize}
			\item The name cycle arboressence comes from the fact that such a digraph looks a lot like an arboressence or rooted out-tree, where the vertices are replaced by cycles (type-(iii) vertices) and the edges by actual arboressences (type-(ii) vertices), also see Figure \ref{Fcyclearbo}. The definition of height is also based on this idea.
			\item 	Note that the root vertex $w$ of a cycle arboressence is unique if it is a type-(i) or type-(ii) vertex, but if it is a type-(iii) vertex it can be any vertex on the path from $w$ to $w$. 
		\end{itemize}
	\end{trs}
	
	Let $v$ be any vertex of the automaton, then by Lemma \ref{Lvcyc} we also have three cases: 
	\begin{itemize}
		\item[(i)'] every walk from $v$ goes to a $0$ and $v$ is labeled by $0$. 
		\item[(ii)'] there is a walk from $v$ to a vertex labeled non-zero, but not a walk from $v$ to $v$.
		\item[(iii)'] there is a walk from $v$ to a vertex labeled non-zero and exactly one walk from $v$ to $v$ that does not contain $v$ a third time.
	\end{itemize}
 	For now a vertex falling in the category (.)' will be called a \textit{case-(.)' vertex}. 
 	
	We will rearrange our given automaton in such a manner that the types of the vertices in Definition \ref{Dcycarbo} correspond to the cases given above while still producing the same automatic sequence. In case-(i)' we can assume that all edges going outwards of $v$ go to $v$ itself. If, furthermore, $v$ has multiple ingoing edges $(., v)$, we may enlarge the $m$-automaton with another case-(i)' vertex $v'$, also having all its outgoing edges ending in $v'$, and replace one of the ingoing edges of $v$ by $(., v')$. This process can be repeated until $v$ has exactly one ingoing edge (or in the case of the `Start'-vertex, none), thus making it a type-(i) vertex as well. To make a case-(ii)' vertex into a type-(ii) vertex we will do something similar, but now instead of enlarging the automaton just by $v'$ we will enlarge it by a copy of everything that can be reached from $v$, including $v$ itself. The only exception is that one should never copy the `Start'--label, since it needs to stay unique. The case-(iii)' vertices are a bit more complicated, because we need to look at all the vertices on the walk from $v$ to $v$ at once. Since the automaton contains no tied vertices, we can assume that the walk from $v$ to $v$ is actually a path or cycle, and that all vertices on this cycle are also case-(iii)' vertices. Additionally, each vertex on the cycle has at least one incoming edge, coming from the cycle. Hence what we want to achieve is that the cycle as an object itself has exactly one incoming edge (or contains the `Start'-vertex). If this is not the case we can copy everything that can be reached from this cycle, including the cycle itself, and redirect one of the incoming edges to the corresponding vertex in the copy. This process can also be repeated until there is exactly one incoming edge in the cycle. Note that the number of copying steps that we have to do is finite, ensuring that the digraph we end up with is also still finite and thus an $m$-automaton. 
	
%	Since a vertex cannot belong to two cycles, we can group case-(iii) vertices together by their cycle. In a similar manner we will group case-(ii) vertices together by arboressences. To do this we might first need to enlarge our $p$-automaton a bit. Suppose two distinct edges arrive at the same vertex or the same cycle and neither of these edges is part of a cycle. Then we take this vertex/cycle and all of the $p$-automaton that can still be reached from this vertex/cycle, and copy it. We attach one of the copies to the first edge and the other copy to the second edge. Now the two edges do not go to the same vertex or cycle any more. This can be done for all such pairs of edges throughout the $p$-automaton, until there are none left. Because there are only finitely many such pairs, we will still have a $p$-automaton in the end, that in fact produces the same automatic sequence. This construction ensures that every case-(ii) vertex has at most one ingoing edge. So, if we group together case-(ii) vertices that can be reached through each other without passing a case-(iii) vertex, these form an arboressence inside the graph. Furthermore, the construction also ensures that each cycle can be reached by at most one sequence of earlier cycles. With that in mind we can view our $p$-automaton as a bigger arboressence, where the vertices are the cycles made of case-(iii) vertices, the edges are the (possibly empty) arboressences made of case-(ii) vertices and the leaves are the case-(i) vertices. Also see Figure \ref{Fcyclearbo}.
	
	\begin{figure}[h]
		\centering
		\includegraphics[width=9.5cm]{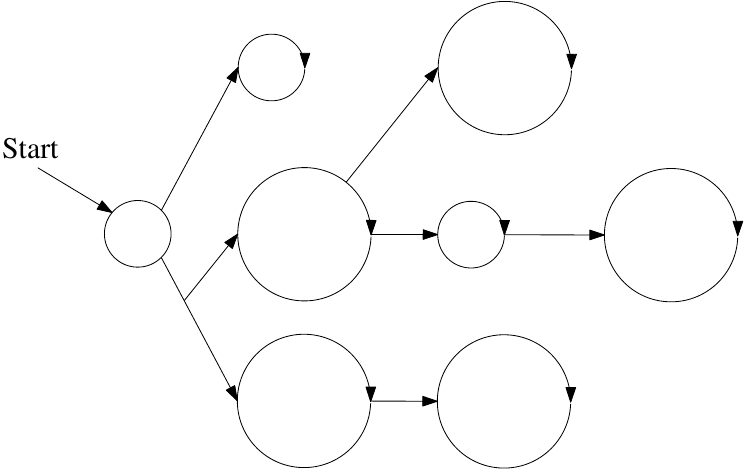}
		\caption{An example of an $m$-automaton without tied vertices viewed as cycle arboressence. The type-(i) vertices and their edges are not drawn.}\label{Fcyclearbo}
	\end{figure}
	
	Before we start with the computational part of the proof, we will introduce some important constants and labels based on the new structure of the $m$-automaton. We define $s$ as the number of cycles there are and label them $C_1, C_2, \ldots, C_s$. For each $C_i$ write $l_i$ for the number of vertices it contains and define $l := \min(l_i)$. Also, let $r$ be the height of the $m$-automaton viewed as cycle arboressence.
	
	Now we will examine how many integers there are in $E((a_i)_{i\in \N})_N$ with $N = m^{k+1}-1$ for some integer $k \geq 1$. Note that the integers $i < m^{k+1}$ are exactly the integers for which the value of $a_i$ can be determined by taking $k+1$ steps in the $m$-automaton. Write $k+1 = xl + y$ with $0 \leq y < l$. Since all the type-(i) vertices are labeled by $0$, it will suffice to give an upper bound on the number of different walks of length $k+1$ that end in a type-(ii) or type-(iii) vertex. \\
	Let $v$ be any type-(ii) or -(iii) vertex. There is a unique sequence of cycles, $C_{i_1}, C_{i_2}, \ldots, C_{i_j}$, with $j\leq r$, the walk has to pass through in order to end up at $v$. We will proceed with induction on $j$. If $j$ is either $0$ or $1$ (where the first can only happen if $v$ is a case-(ii) vertex), the walk to $v$ of length $k+1$ must be either unique or non-existent. Now suppose for some $j \geq 1$ we already know that there are $O(\log(N)^{j-2})$ walks to some vertex $u$ in the cycle $C_{i_{j-1}}$ (induction hypothesis). To extend such a walk to $v$ we can either take the shortest walk to $C_{i_j}$ and cycle in $C_{i_j}$ until our steps are gone, or we can cycle $1, 2, \ldots$ up to $x$ times in $C_{i_{j-1}}$ and then go to $C_{i_j}$. This gives  at most $(x+1)O(\log(N)^{j-2})$ ways to end up at $v$. We had $N = m^{lx + y}$, so $x+1 = O(\log(N))$ and, by induction, we get that there are $O(\log(N)^{j-1})$ walks of length $k+1$ to $v$. Since there are at most $r$ cycles we can pass through consecutively in our $m$-automaton we get $E((a_i)_{i\in \N})_N = O(\log(N)^{r-1})$ and $(a_i)_{i\in \N}$ is sparse.
\end{proof}

\begin{trs}~
	\begin{itemize}
		\item 	Theorem \ref{Tautosparse} can also be used to see if some element $(a_i)_{i\in \N}$ of $A^{\N}$ is an automatic sequence. Namely, if  $\# E((a_i)_{i\in \N})_N$ is not $O(\log(N)^r)$ for any $r$ and there also does not exist any $\alpha$ such that $\# E((a_i)_{i \in \N})_N \geq N^\alpha$ for sufficiently large $N$, then an $m$-automaton for $(a_i)_{i\in \N}$ both has to and cannot contain a tied vertex, so it does not exist. \\
		For an example of this, define $(a_i)_{i\in \N}$ by $a_0=0, a_1 = 1, a_2 = 0 , a_3=0$ and for $N>3$, $a_N = 0$ if and only if $\lceil \log(N)^{\log(\log(N))} \rceil = \lceil \log(N-1)^{\log(\log(N-1))}\rceil$. Then $E((a_i)_{i\in \N})_N$ grows as $\log(N)^{\log(\log(N))}$ which is neither $O(\log(N)^r)$ nor grows as fast as $N^\alpha$ for any $\alpha > 0$. So $(a_i)_{i\in \N}$ meets the requirements and cannot be automatic. 
	%I think this even means it is transcendental. 
	
	%	\begin{align*}
		%	\log(N)^{\log(\log(N))} &\geq CN^\alpha \\
		%	\log(\log(N)^{\log(\log(N))}) &\geq \log(CN^\alpha) \\
		%	\log(\log(N))^2 &\geq \log(C) + \alpha\log(N) \\
		%	x^2 &\geq \log(C)+\alpha e^{e^{x}}.
		%	\end{align*}
	%	Where $x  = \log(\log(N))$. This does not hold for sufficiently large $x$ and thus $N$. 	
	\item One might expect that in the sparse case $E((a_i)_{i\in \N})_N$ grows at least like $\log(N)^r$, because that is, up to a constant, the total number of walks from `Start' of length $k$. However, the last cycle of the cycle arboressence will in practice always consist of zero-edges and its outgoing edges will all go to type-(i) vertices. This can also be seen in Figure \ref{FautoF2.a}.
\end{itemize}
\end{trs}
\subsection{Determining the rank of sparseness}
The structure of a cycle arboressence does not only make it easy to determine a lower bound for the support growth, but also gives us a new way to determine exactly what the rank of the automatic series is. 
\begin{tx}\label{Trank}
	Let $(a_i)_{i\in \N}$ be a sparse automatic series in a set $A$ with cycle arboressence of height $r$. Then the rank of sparseness of $(a_i)_{i\in \N}$ is $r-1$. 
\end{tx}
\begin{proof}
	From the proof of Theorem \ref{Tautosparse} we know that $\#E((a_i)_{i\in \N})_N = O(\log(N)^{r-1})$, so the rank of sparseness of $(a_i)_{i\in \N}$ is at most $r-1$. Now we need to show that the rank cannot be smaller, which we will do by a similar induction argument and thus we will use the same notation. In particular, recall that $j$ is the number of cycles one has to partially pass through to reach $v$.
	
	The claim we are going to prove is that for all $C , \varepsilon > 0$ there exists an $M \in \N$ such that if $k > M$ is also an integer, the number of walks from `Start' to $v$ in $k$ steps is larger than $C\cdot k^{j-1-\varepsilon}$. The statement already holds for $j = 1$, since in that case the number of walks is always equal to either $0$ or $1$. Now assume $j-1 \geq 1$ and that for all vertices $u$ in the cycle $C_{i_{j-1}}$ the lower bound $C\cdot k^{j-2-\varepsilon}$ holds whenever $k> M$. Define $z$ as the minimal number of steps one has to take to go from any $u$ to $v$ and let $k' = xl + z + M$ with $x$ sufficiently large. The number of walks to $v$ is then at least
	$$C\cdot (l^{j-2-\varepsilon} + (2l)^{j-2-\varepsilon} + \ldots + (xl)^{j-2-\varepsilon}) \geq \frac{C}{l(j-1-\varepsilon)} (xl)^{j-1-\varepsilon} \geq \frac{C}{2l(j-1-\varepsilon)} (k')^{j-1-\varepsilon},$$
	where the second inequality uses that $x$ is sufficiently large. This proves that the statement also holds for a vertex $v$ in the $j$th cycle, though the integer $M$ might be larger than before. 
	
	To finish the proof we look at some non-zero vertex $v$ that one can only reach by (partially) going through $r$ cycles. This vertex exists because a type-(iii) vertex still has to be a case-(iii)' vertex. Fix $C, \varepsilon > 0$ and take $N = m^k$ with $k > M$, we find that $\# E((a_i)_{i\in \N})_N \geq C\log(N)^{r-1-\varepsilon}$.  Since for each choice of $C, \varepsilon$ we find arbitrarily large $N$ for which this holds, we can conclude that $E((a_i)_{i\in \N}) \neq O(\log(N)^{r-1-\varepsilon})$ for any $\varepsilon$ larger than $0$ and thus the rank of sparseness cannot be smaller than $r-1$. 
\end{proof}
\begin{tc}
	If $(a_i)_{i\in \N}$ is an automatic sparse sequence, then its rank of sparseness (defined by an infimum) is attained by an integer.
\end{tc}
\begin{proof}
	From Theorem \ref{Trank} we know the rank is the height of a cycle arboressence minus $1$, which is an integer. In the proof of Theorem \ref{Tautosparse} we saw  that if $r$ is the height of the cycle arboressence, then $\#E((a_i)_{i\in \N})_N = O(\log(N)^{r-1})$. So indeed, the rank is attained.
\end{proof}
\begin{te}
	In Figure \ref{FautoF2.a} we saw an automaton for a sparse sequence. This automaton is already in the cycle arboressence form with two cycles (the self-loops at the top vertices). These cycles can be used consecutively, so the height is $2$ and we can conclude that the rank of sparseness is $1$. 
\end{te}
\begin{tr}
	The proof of Szilard, Yu and Zhang \cite{szilard1992regularlanguages} that shows that the rank of an automatic sequence is always an integer uses the notion of simple sparse sets. It states that a sequence has rank $r$ if its support can be written as a union of disjoint simple sparse sets of rank at most $r$. It is possible to go from the cycle arboressence to simple sparse sets and back again. For the first direction, let $a_i$ be the word of the walk between either `Start' and the first cycle $(i=0)$, the cycle $C_{i_j}$ and some non-zero vertex ($i=j$) or between consecutive cycles ($1\leq i \leq j-1$) and let $b_i$ be the word of a cycle. Then the support contains the simple sparse set $\{a_{j}b_{j}^{x_{j}}\cdots a_{1}b_1^{x_1}a_0 \mid x_1, \ldots, x_r \in \Z_{\geq 0} \}$. 
\end{tr}

\section{An upper bound for non-sparse sequences}
\subsection{Bounds on the number of walks}\label{SBoundsnumberwalks}
In this subsection we state the definitions and lemmas we will need to give an upper bound for $\alpha$ in the next subsection. The focus will lie on bounding the number of closed walks with specified length through a tied vertex. Intuitively it makes sense that $\alpha$ depends on this; if there are more closed walks of specified length we can construct more walks via the tied vertex to a non-zero vertex. 
\begin{td}
	Let $u$ and $v$ be vertices of a (di)graph and $k$ a non-negative integer. Define $\Omega_k(u, v)$ as the number of walks of length $k$, starting at $u$ and ending at $v$. When $u = v$ we also write $\Omega_k(v)$. Furthermore, the number of walks of length less or equal to $k$ is denoted by $\Omega_{\leq k}(u, v)$, that is $\Omega_{\leq k}(u, v) := \sum_{i=0}^k \Omega_i(u, v).$
\end{td}

In Lemma \ref{Ladjmat} we recalled that the number of walks of certain length in a (di)graph can be determined by its adjacency matrix. Namely, if $u$ and $v$ are two vertices of some (di)graph with adjacency matrix $A$, then the number of walks of length $k$ from $u$ to $v$ corresponds to $(A^k)_{u, v}$. Therefore, $\Omega_k(u, v) = (A^k)_{u, v}$. 

Now we want to use Perron-Frobenius theory, but this cannot be done immediately. This is because, given any $m$-automaton or digraph, the adjacency matrix is not necessarily irreducible. In fact, by Lemma \ref{Lstrconnirr}, the adjacency matrix is irreducible if and only if it corresponds to a strongly connected digraph. For each vertex $v$ of a certain $m$-automaton we will be able to look at a subdigraph that is strongly connected. The idea is based on the observation that any vertex on a walk from $v$ to $v$ must be strongly connected to $v$. Indeed, suppose $W$ is some walk from $v$ to $v$ and $u$ any vertex on $W$, then clearly there is a path going from $v$ to $u$ and a path from $u$ to $v$. Even better, if $w$ is yet another vertex strongly connected to $v$, we can also construct a path from $u$ to $w$ and $w$ to $u$ by going via $v$. What we see here is that all vertices involved in walks from $v$ to $v$ are strongly connected to each other. It now makes sense to look at the subdigraph that only consists of vertices that are strongly connected to $v$ and the edges between them. 
\begin{td}
	Let $D=(V, E)$ be a digraph and $v \in V$. Define 
	\begin{itemize}
		\item the set of vertices $V_v := \{u\in V \mid \textup{there are paths from } u \textup{ to } v \textup{ and from } v \textup{ to } u\}$,
		\item the set of edges $E_v := \{(u, w) \in E \mid u, w \in V_v\}$. 
	\end{itemize}
	Then $D_v := (V_v, E_v)$ is called the \textit{strongly connected subdigraph of $D$ at $v$}. Its adjacency matrix will be denoted by $A_v$, its spectral radius by $\rho_v$ and the positive normalized right eigenvector belonging to $\rho_v$ by $\vec{x}_v$. The spectral radius will also be called \textit{the spectral radius of $v$}.
\end{td}
\begin{tr}
	Note that previous definition makes sense, since by Proposition \ref{PPerronFrobenius} such an $\vec{x}_v$ exists and is unique. 
\end{tr}
For the next proof, it is important to note that the number of walks from $v$ to $v$ of length $n$ in $D_v$ is the same as in $D$. This follows from the observation we made before that any such walk must consist solely of vertices that are strongly connected to $v$. Therefore we still have that $\Omega_n(v) = (A_v^n)_{v, v}$. 

\begin{tx} \label{TBoundOmega}
	Let $D = (V, E)$ be some $m$-automaton and $v$ any tied vertex of $D$. Then there exists some constant $C > 0$ so that 
	\begin{enumerate}[(i)]
		\item for all $n \in \N$ we have $\Omega_n(v) \leq \rho_v^n$.
		\item for infinitely many $n \in \N$ we have $\Omega_n(v) \geq C\rho_v^n$.
	\end{enumerate}
\end{tx}

\begin{proof}
	Consider the strongly connected subdigraph $D_v$ of $D$ at $v$. By Perron-Frobenius we get the equation $A_v\vec{x}_v = \rho_v\vec{x}_v$, which extends to $A_v^n\vec{x}_v = \rho_v^n\vec{x}_v$ for any $n \in \Z_{\geq 0}$. Because we are interested in the walks from $v$ to $v$, we study the row corresponding to vertex $v$ of this set of linear equations. Without loss of generality we may assume this to be the first row and write $u_2, \ldots, u_{s}$ for the other vertices in $V_v$. This gives us the following expression:
	$$\Omega_n(v)x_{v, 1} + \Omega_n(v, u_2)x_{v, 2} + \ldots + \Omega_n(v, u_{s})x_{v, s} = \rho_v^nx_{v, 1}.$$	
	Since $x_{v, i} > 0$ and $\Omega_n(v, u_i) \geq 0$ for all $i$, it is immediate that $\Omega_n(v) \leq \rho_v^n$ for all $n$, proving \emph{(i)}. The inequality in \emph{(ii)} needs a bit more consideration. By the pigeonhole principle, we get that for all $n$ we either have $\Omega_n(v)x_{v, 1} \geq \frac 1m\rho_v^nx_{v, 1}$ or there is some $i\geq 2$ such that $\Omega_n(v, u_i)x_{v, i} \geq \frac 1m\rho_v^nx_{v, 1}$. For the second case we look at some path $P_i$ from $u_i$ to $v$. This path must exist, since the digraph is strongly connected. It should be clear that $\Omega_{n+l(P_i)}(v) \geq \Omega_n(v, u_i)$ which then leads to $$\Omega_{n+l(P_i)}(v) \geq \left(\frac{x_{v, 1}}{sx_{v, i}}{\rho_v^{-l(P_i)}}\right)\rho_v^{n + l(P_i)}.$$
	The path $P_i$ does not depend on the value of $n$ and therefore we can fix a choice for each $2 \leq i \leq s$. Furthermore, let $P_1$ be the empty path. Taking the constant $$C := \min_{1 \leq i \leq s}\left(\frac{x_{v, 1}}{sx_{v, i}}{\rho_v^{-l(P_i)}}\right)>0$$ we indeed get that for infinitely many $n$,
	\[\Omega_n(v)\geq C\rho_v^n.\qedhere \]
\end{proof}
\begin{tr}
	While $\Omega_n(v) \geq C\rho_v^n$ does not necessarily hold for all $n$, but it does hold for at least one value between $n$ and $n + \max_{1 \leq i \leq s}(l(P_i))$. 
\end{tr}	

\begin{tc}
	Let $v$ be a tied vertex of some $m$-automaton, then its spectral radius $\rho_v$ is bigger than $1$.
\end{tc}
\begin{proof}
	Since $v$ is tied, we know there is some $N \in \N$ for which $\Omega_N(v) \geq 2$. Theorem \ref{TBoundOmega} now implies $\rho_v \geq \sqrt[N]{2} > 1$. 
\end{proof}
\begin{tc}\label{CsupB}
	For any tied vertex $v$ we have 
	$$\sup_{n\in\N}\left(\frac{\log_m(\Omega_n(v))}{n}\right) = \log_m(\rho_v).$$
\end{tc}
\begin{proof}
	From Theorem \ref{TBoundOmega} we get 
	$$\sup_{n\in\N}\left(\frac{\log_m(\rho_v^n)}{n}\right) \geq \sup_{n\in\N}\left(\frac{\log_m(\Omega_n(v))}{n}\right) \geq \sup_{n\in\N}\left(\frac{\log_m(C\rho_v^n)}{n}\right).$$
	The left hand side equals $\log_m(\rho_v)$ and the right hand side $\sup_{n\in\N}(\frac{\log_m(C)}n) + \log_m(\rho_v)$. Since $n$ gets arbitrarily large and $\log_m(C)$ is a negative constant, the first term is equal to zero. So both the lower and upper bound are $\log_m(\rho_v)$, which proves the corollary.
\end{proof}
\begin{tl} \label{Latmostntoexactn}
	Let $D = (V, E)$ be an $m$-automaton, $v_1, \ldots, v_s$ tied vertices in $V$ and $k$ some positive integer. Define $\rho := \max_{1\leq j \leq s}(\rho_{v_j})$ and let $K_k \subset \Z_{\geq 0}^s$ consist of the solutions $\vec{k} = (k_1, \ldots, k_s)$ to $k_1 + \ldots + k_s = k$. Then, for all $k \in \N$, the number $\mathcal W_k$, defined as \[ \mathcal{W}_k := \sum_{\vec{k} \in K_k} \prod_{i=1}^s \Omega_{\leq k_i}(v_i),\] satisfies $\mathcal W_k = O(k^{s-1}\rho^k),$ where the implied constant depends on $\rho$.
\end{tl}
\begin{proof}
	We will prove this lemma by induction on the number of tied vertices $s$. To keep the notation uncluttered we assume $\rho_{v_1}\geq \ldots \geq \rho_{v_s}$ and therefore $\rho = \rho_{v_1}$. This can be done without loss of generality.
	
	For $s = 1$ we get that $\mathcal{W}_k = \Omega_{\leq k}(v_1) = \sum_{i=0}^k \Omega_k(v_1)$. Theorem \ref{TBoundOmega} provided $\rho^i$ as an upper bound for each $\Omega_i(v_1)$ so that
	$$\mathcal{W}_k \leq \frac{\rho^{k+1} -1}{\rho -1} \leq \frac{\rho}{\rho - 1}\rho^k = O(\rho^k).$$
	The last step uses $\rho > 1$. This proves the induction basis. 
	
	Now suppose we know the lemma holds for any $s-1$ tied vertices (induction hypothesis). We will prove that it also holds for $s$ tied vertices. First we take out the last factor of the product in $\mathcal{W}_k$ to get
	\[\sum_{k_0=0}^k \Omega_{\leq k_0}(v_s) \sum_{\vec{k}\in K_{k-k_0}} \prod_{i=1}^{s-1}\Omega_{\leq k_i}(v_i),\] 
	where $K_{k-k_0}$ is now a subset of $\Z_{\geq 0}^{s-1}$ instead of $\Z^s_{\geq 0}$. For each $k_0$ we can bound both factors of the main summand, respectively by $O(\rho^k_0)$ and $O(k^{s-2}\rho^{k-k_0})$, using the induction basis and induction hypothesis. This yields
	$$\mathcal{W}_k = \sum_{k_0=0}^kO(k^{s-2}\rho^k) = O(k^{s-1}\rho^k).$$
	By the principle of induction, this proves the claim.
\end{proof}
\begin{tc}\label{Catmostntoexactn}
	There exists a constant $C'>0$ so that, for infinitely many $k$ and some $1 \leq j \leq s$,
	\[\mathcal{W}_k \leq C'k^{s-1}\Omega_n(v_j).\]
\end{tc}
\begin{proof}
	Theorem \ref{TBoundOmega} shows that for any tied vertex $v$ there are infinitely many $k$ such that $\rho_v^k \leq \frac 1C \Omega_k(v)$. The result now follows directly from Lemma \ref{Latmostntoexactn}.
\end{proof} 

\subsection{An upper bound for non-sparse series}
In this subsection we will prove two theorems. The first gives a supremum of all values $\alpha$ for which $\# E((a_i)_{i \in \N})_N \geq N^\alpha$ holds. The second shows that this supremum can be expressed as the logarithm of an integral root of a Perron number. To state and prove these results we will use the theory developed in Subsection \ref{SBoundsnumberwalks}.

\begin{tx} \label{Tupperalpha}
	Let $D=(V, E)$ be an $m$-automaton, $V'\subset V$ the set of tied vertices and $(a_i)_{i \in \N}$ its automatic sequence. Define \[B:= \max_{v \in V'} (\log_m(\rho_v)).\]
	Then \[\sup\{\alpha > 0 : \#E((a_i)_{i \in \N})_N \geq N^\alpha \textup{ for sufficiently large $N$}\} = B.\]
\end{tx}
\begin{proof}
	In the case that $V'$ is empty we get an empty maximum as definition for $B$. However, the supremum over the allowed $\alpha$ is empty as well, since $(a_i)_{i \in \N}$ will be sparse. So from now on we may assume $V'$ is non-empty.
	
	We will first prove that $B$ gives an upper bound for the values that $\alpha$ can take. Secondly we will show that $B$ is indeed the best we can do to give such an upper bound. To prove these two things we heavily rely on the fact that, by Corollary \ref{CsupB}, $B$ is equal to 
	\begin{equation}\label{Bsup}
		\sup_{x\in\N, v \in V'}\left(\frac{\log_m(\Omega_x(v))}{x}\right).
	\end{equation} 
	
	Let $\alpha$ be any value less than $B$. By (\ref{Bsup}) we know that there must exist some $v \in V'$ and $x \in \N$ such that $\alpha \leq \frac{\log_m(\Omega_x(v))}{x} \neq 0$. We follow the same proof as for the first part of Theorem \ref{Tautosparse}, except now we have $\Omega_x(v)$ walks, $W_1, \ldots, W_{\Omega_x(v)}$, to choose from for $W_{i_1}, \ldots, W_{i_k}$. Therefore, there are at least $$\frac{(\Omega_x(v))^{k}-1}{\Omega_x(v)-1}$$ distinct values in $\# E((a_i)_{i \in \N})_N$ if $N = m^{(k-1)x+y+z}$. Since $\Omega_x(v)$ does not equal $1$, we can choose $k$, and therefore $N$, large enough such that $\#E((a_i)_{i \in \N})_N \geq N^\alpha$.
	
	Now suppose that $\alpha > B$, then it suffices to prove that there exist arbitrarily large $N$ such that $N^\alpha \geq \# E((a_i)_{i \in \N})_N$. In other words, there must be infinitely many such $N$. We will show that any $N = m^k$, where $k$ meets the condition of Corollary \ref{Catmostntoexactn} and is sufficiently large, has this property. 
	
	The proof will be by contradiction. So assume $\alpha = B + \varepsilon$ for some $\varepsilon > 0$ and $\#E((a_i)_{i \in \N})_N \geq N^\alpha$ for all $N$ sufficiently large. We know from (\ref{Bsup}) that $B \geq \frac{\log_m(\Omega_x(v))}{x}$ holds for all $x \in \N$ and $v\in V'$. In particular it holds for $x = k$. Now let $N = m^k$ and $k$ sufficiently large, then we can rewrite this as
	$$\# E((a_i)_{i \in \N})_N \geq m^{kB}N^{\varepsilon} \geq \Omega_k(v) N^\varepsilon.$$
	
	On the other hand we can also give an upper bound for $\# E((a_i)_{i \in \N})_N$. Recall that $\# E((a_i)_{i \in \N})_N$ counts the walks of length $k$ that start at `Start' and end in a non-zero vertex. Any such walk has a certain structure, described below, and we will be able to estimate the number of walks with this structure. Let $s\geq 0$ be an integer and $v_1, \ldots, v_s \in V'$ all distinct. We define the following sets:
	\begin{enumerate}[(i)]
		\item the set $\Omega(\textup{Start}, v_1)(V')$ consisting of walks from `Start' to $v_1$ not passing through any tied vertex.
		\item for $1 \leq i \leq s$ the set $\Omega(v_i)$ consisting of walks from $v_i$ to $v_i$.
		\item for $1 \leq i \leq s-1$ the set $\Omega(v_i, v_{i+1})(V')$ consisting of walks from $v_i$ to $v_{i+1}$ not passing through any tied vertex.
		\item the set $\Omega(v_s, \textup{End})(V')$ consisting of walks from $v_s$ to any non-zero vertex not passing through any tied vertex.
	\end{enumerate}
	In Figure \ref{Fnonsparse} the sets are visualized for $s=2$.
	
		\begin{figure}[h]
		\centering
		\includegraphics[width=16cm]{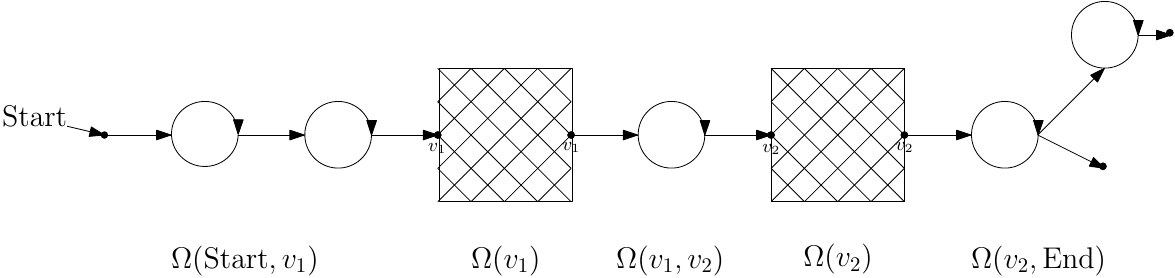}
		\caption{Possible paths from `Start' to some non-zero vertex while passing through $v_1$ and $v_2$, both tied vertices. The sets $\Omega(\ldots)$ consist of all possible paths for their specific part.} \label{Fnonsparse}
	\end{figure}

	Now let $W$ be any walk from `Start' to a non-zero vertex. There exists some integer $s \geq 0$ and $v_1, \ldots, v_s$ distinct in $V'$ for which we can take a unique element of all above defined sets such that $W$ is a composition of these walks. Note that there is only one possibility for a composition, since the begin and end vertex of each walk are known. However, the choice of $v_1, \ldots, v_s$ is not necessarily unique, but we will only need existence. %since there might be some tied vertices $u$ and $v$ such that there exists a walk from $u$ to $v$ and one from $v$ to $u$. This will not be important for the rest of the proof.
	
	An $m$-automaton has a finite number of vertices and therefore of tied vertices. This implies there are at most $(|V'| + 1)!$ choices for $s$, $v_1, \ldots, v_s$. Since we are only interested in walks of length $k$, one of those combinations must give the largest number of walks of length $k$. From now on we write $s$, $v_1$, \ldots, $v_s$ for this combination. We will continue by estimating the number of walks that can be written as a composition of elements of the sets defined before.
	
	Recall from the proof of Theorem \ref{Tautosparse} that either the number of walks between two vertices is of order $O(\log(N)^r)$ for some integer $r$, or the $m$-automaton $D$ includes a tied vertex. This gives an estimate for the cardinality of the sets of (i), (iii) and (iv) when we also require the length to be at most $k$. This leaves $k' \leq k$ steps for the walks in the sets of (ii). Note that the number of elements of length at most $k_i$ for any such set equals $\Omega_{\leq k_i}(v_i)$. Defining $K_k := \{(k_1, \ldots, k_s) \subset \Z_{\geq 0}^s \mid k_1 + \ldots + k_s = k \}$, we get
	\[\# E((a_i)_{i \in \N})_N \leq (|V'|+1)! O(\log(N)^r)\sum_{\vec{k} \in K_k}\prod_{i=1}^s \Omega_{\leq k_i}(v_i).\]
	
	Now we can use Corollary $\ref{Catmostntoexactn}$ to see that \[\# E((a_i)_{i \in \N})_N \leq C k^{s-1}\Omega_k(v) \log(N)^r = C \Omega_k(v)\log(N)^{r+s-1}\] for some tied vertex $v$, constant $C>0$ and infinitely many $k$. Putting everything together we have 
	\[\Omega_k(v)N^\varepsilon < \# E((a_i)_{i \in \N})_N \leq C\Omega_k(v)\log(N)^{r+s-1},\]
	for infinitely many $N = m^k$. However, since $\log(N)^{r+s-1}$ will grow slower than $N^\varepsilon$, there will exist some sufficiently large $N$ such that $\Omega_k(v)N^{\varepsilon} > C\Omega_k(v)\log(N)^{r+s-1} \geq \# E((a_i)_{i \in \N})_N$, a contradiction. We conclude that there does not exist a sufficiently large $N$ such that for all $n \geq N$ it holds that $\# E((a_i)_{i \in \N})_n \geq n^\alpha$ if $\alpha$ is bigger than $B$. Hence $B$ is indeed the supremum of all $\alpha$ for which this condition does hold. 
\end{proof}

\begin{tr}
	Theorem \ref{Tupperalpha} tells us that any $\alpha < B$ can be used in Theorem \ref{Tautosparse}(i) and that no $\alpha > B$ can. However, it says nothing about $\alpha = B$. As can be seen in Figures \ref{FalphaisB} and \ref{FalphaisnotB}, either situation occurs. In both examples we have $B = \frac{\log_2(2)}{2} = \frac 12$ and in the first we see $\# E((a_i)_{i \in \N})_N \geq N^{\frac 12}$ for all $N \geq 1$, whereas in the second $\# E((a_i)_{i \in \N})_N < N^{\frac 12}$ for infinitely many $N$. 
	
	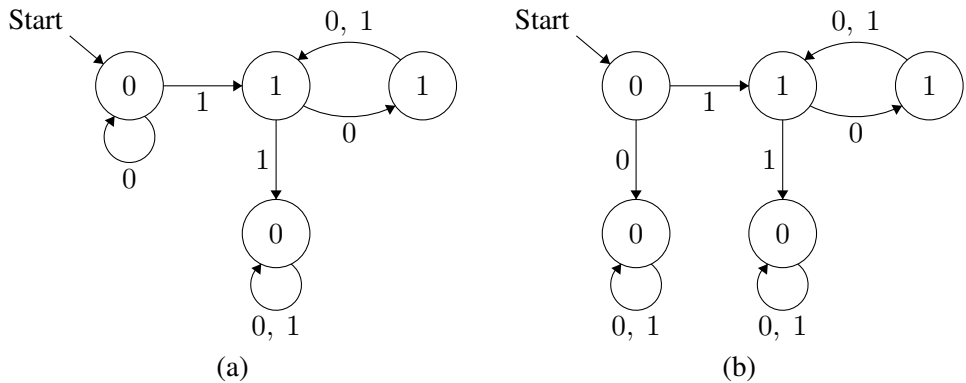
\begin{figure}[h] 
		\centering
		\subfloat[\label{FalphaisB}]{\centering \begin{tikzpicture}[scale=0.15]
				\tikzstyle{every node}+=[inner sep=0pt]
				\draw [black] (17.7,-13.5) circle (3);
				\draw (17.7,-13.5) node {$0$};
				\draw [black] (30.7,-13.5) circle (3);
				\draw (30.7,-13.5) node {$1$};
				\draw [black] (43.7,-13.5) circle (3);
				\draw (43.7,-13.5) node {$1$};
				\draw [black] (30.7,-26.5) circle (3);
				\draw (30.7,-26.5) node {$0$};
				\draw [black] (12.4,-9.1) -- (15.39,-11.58);
				\draw (11.85,-7.66) node [left] {Start};
				\fill [black] (15.39,-11.58) -- (15.1,-10.69) -- (14.46,-11.46);
				\draw [black] (19.023,-16.18) arc (54:-234:2.25);
				\draw (17.7,-20.75) node [below] {$0$};
				\fill [black] (16.38,-16.18) -- (15.5,-16.53) -- (16.31,-17.12);
				\draw [black] (20.7,-13.5) -- (27.7,-13.5);
				\fill [black] (27.7,-13.5) -- (26.9,-13) -- (26.9,-14);
				\draw (24.2,-14) node [below] {$1$};
				\draw [black] (41.284,-15.255) arc (-63.41702:-116.58298:9.126);
				\fill [black] (41.28,-15.26) -- (40.34,-15.17) -- (40.79,-16.06);
				\draw (37.2,-16.72) node [below] {$0$};
				\draw [black] (32.636,-11.235) arc (127.91326:52.08674:7.428);
				\fill [black] (32.64,-11.23) -- (33.57,-11.14) -- (32.96,-10.35);
				\draw (37.2,-9.17) node [above] {$0,\mbox{ }1$};
				\draw [black] (30.7,-16.5) -- (30.7,-23.5);
				\fill [black] (30.7,-23.5) -- (31.2,-22.7) -- (30.2,-22.7);
				\draw (30.2,-20) node [left] {$1$};
				\draw [black] (32.023,-29.18) arc (54:-234:2.25);
				\draw (30.7,-33.75) node [below] {$0,\mbox{ }1$};
				\fill [black] (29.38,-29.18) -- (28.5,-29.53) -- (29.31,-30.12);
		\end{tikzpicture}}
		\qquad
		\subfloat[\label{FalphaisnotB}]{\centering \begin{tikzpicture}[scale=0.15]
				\tikzstyle{every node}+=[inner sep=0pt]
				\draw [black] (12.4,-9.1) -- (15.39,-11.58);
				\draw (11.85,-7.66) node [left] {Start};
				\fill [black] (15.39,-11.58) -- (15.1,-10.69) -- (14.46,-11.46);
				\draw [black] (17.7,-13.5) circle (3);
				\draw (17.7,-13.5) node {$0$};
				\draw [black] (17.7,-26.5) circle (3);
				\draw (17.7,-26.5) node {$0$};
				\draw [black] (30.7,-13.5) circle (3);
				\draw (30.7,-13.5) node {$1$};
				\draw [black] (43.7,-13.5) circle (3);
				\draw (43.7,-13.5) node {$1$};
				\draw [black] (30.7,-26.5) circle (3);
				\draw (30.7,-26.5) node {$0$};
				\draw [black] (20.7,-13.5) -- (27.7,-13.5);
				\fill [black] (27.7,-13.5) -- (26.9,-13) -- (26.9,-14);
				\draw (24.2,-14) node [below] {$1$};
				\draw [black] (30.7,-16.5) -- (30.7,-23.5);
				\fill [black] (30.7,-23.5) -- (31.2,-22.7) -- (30.2,-22.7);
				\draw (30.2,-20) node [left] {$1$};
				\draw [black] (17.7,-16.5) -- (17.7,-23.5);
				\fill [black] (17.7,-23.5) -- (18.2,-22.7) -- (17.2,-22.7);
				\draw (17.2,-20) node [left] {$0$};
				\draw [black] (19.023,-29.18) arc (54:-234:2.25);
				\draw (17.7,-33.75) node [below] {$0,\mbox{ }1$};
				\fill [black] (16.38,-29.18) -- (15.5,-29.53) -- (16.31,-30.12);
				\draw [black] (32.023,-29.18) arc (54:-234:2.25);
				\draw (30.7,-33.75) node [below] {$0,\mbox{ }1$};
				\fill [black] (29.38,-29.18) -- (28.5,-29.53) -- (29.31,-30.12);
				\draw [black] (41.284,-15.255) arc (-63.41702:-116.58298:9.126);
				\fill [black] (41.28,-15.26) -- (40.34,-15.17) -- (40.79,-16.06);
				\draw (37.2,-16.72) node [below] {$0$};
				\draw [black] (32.636,-11.235) arc (127.91326:52.08674:7.428);
				\fill [black] (32.64,-11.23) -- (33.57,-11.14) -- (32.96,-10.35);
				\draw (37.2,-9.17) node [above] {$0,\mbox{ }1$};
		\end{tikzpicture}}
		\caption{Two $2$-automata for which the corresponding sequence $(a_i)_{i \in \N}$ has either (a) $\# E((a_i)_{i \in \N})_N \geq N^{\frac12} = N^B$ for all $N$ or (b) $\# E((a_i)_{i \in \N})_N < N^{\frac12} = N^B$ for infinitely many $N$. In case (a) this can be seen by noting $\# E((a_i)_{i \in \N})_{4^k+1} = 2^{k+1}= \sqrt{4^{k+1}}$ for all $k \in \Z_{\geq 0}$, and in case (b) by noting $\# E((a_i)_{i \in \N})_{4^k} = 2^{k-1}$ for all $k \in \Z_{\geq 0}$.}\label{Fsupatt}
	\end{figure}
\end{tr}
\begin{te}
	In Figure \ref{FautoF2.b} we saw an automaton for a non-sparse sequence. This automaton has three tied vertices that all have a spectral radius of $\sqrt 2$. Hence, $B = \log_2(\sqrt 2) = \frac 12$. 
\end{te}
About non-negative integral matrices and their spectral radii the following result is known due to D. Lind.

\begin{td}
	Let $\rho > 1$ be an algebraic integer in $\R$ such that all other roots of its minimal polynomial have smaller absolute value than $|\rho|$. Then $\rho$ is called a \textit{Perron number}.
\end{td}

\begin{tx}[\cite{lind1983entropies}, Theorem 3] \label{Tlind}
	A positive number is the spectral radius of an irreducible non-negative integral matrix if and only if some positive integral power of it is a Perron number.
\end{tx}

Since we also work with irreducible non-negative matrices we can apply this result to the value $\beta := m^B$.

\begin{tp}\label{Pintegral}
	Let $B$ be defined as in Theorem \ref{Tupperalpha}. Then
	\begin{enumerate}[(i)]
		\item $B$ can be written as $\log_m(\beta)$ with $\beta$ an integral root of a Perron number.
		\item $B$ is either transcendental or rational.
	\end{enumerate} 
\end{tp}
\begin{proof}
	The first claim follows directly from Theorems \ref{Tupperalpha} and \ref{Tlind}.
	
	For the second part we use the Gelfond-Schneider Theorem of which a proof can be found in \cite[Appendix~1, Corollary~2]{lang2002algebra}. It states that for any two algebraic numbers $a, b$ such that $a \neq 0, 1$ and $b$ is irrational, $a^b$ is transcendental. Since $m^B = \beta$ is algebraic, we find that $B$ cannot be an irrational algebraic number and the second claim follows.
\end{proof}
\begin{tr}
	Theorem \ref{Tlind} also says that all integral roots of Perron numbers can be obtained as the spectral radius of some irreducible non-negative matrix. This also implies that each integral root of a Perron number can occur as $\beta$, though it might not be possible to do this for every $m$ since we have the restriction of having at most $m$ outgoing edges. 
\end{tr}

\subsection*{Acknowledgments}
The automata in this paper were drawn using the Finite State Machine Designer by Evan Wallace \cite{drawautomatons}. 
\addcontentsline{toc}{chapter}{Bibliography}  
\printbibliography
\end{document}